\documentclass[11pt,a4paper]{article}

\usepackage{amsmath,amssymb,amsfonts,amsthm}\allowdisplaybreaks
\usepackage{tikz,float}
\usepackage{enumerate}
\theoremstyle{plain} 
\newtheorem{theorem}{Theorem}[section]
\newtheorem{lemma}[theorem]{Lemma}
\newtheorem{corollary}[theorem]{Corollary}
\theoremstyle{definition}

\newtheorem{assumptionDis}{Assumption}

\newtheorem*{assumption_iad}{Assumption (IAD)}
\newtheorem*{assumption_w}{Assumption (W)}
\newtheorem*{assumption_m}{Assumption (M)}
\newtheorem*{assumption_FVC}{Assumption (FVC)}
\newtheorem*{assumption_pos}{Assumption (Pos)}
\theoremstyle{remark}
\newtheorem{remark}[theorem]{Remark}
\renewcommand{\epsilon}{\varepsilon}
\renewcommand{\Im}{\operatorname{Im}}

\newcommand{\RR}{\mathbb{R}}

\newcommand{\NN}{\mathbb{N}}	
\newcommand{\ZZ}{\mathbb{Z}}

\newcommand{\BIGOP}[1]{\mathop{\mathchoice%
{\raise-0.22em\hbox{\huge $#1$}}%
{\raise-0.05em\hbox{\Large $#1$}}{\hbox{\large $#1$}}{#1}}}
\newcommand{\BIGboxplus}{\mathop{\mathchoice%
{\raise-0.35em\hbox{\huge $\boxplus$}}%
{\raise-0.15em\hbox{\Large $\boxplus$}}{\hbox{\large $\boxplus$}}{\boxplus}}}
\newcommand{\bigtimes}{\BIGOP{\times}}
\usepackage{microtype}
\begin{document}
\title{Minami's estimate: beyond rank one perturbation and monotonicity}

\author{Martin Tautenhahn$^1$ and Ivan Veseli\'c$^1$}

\date{}

\footnotetext[1]{Technische Universit\"at Chemnitz, Fakult\"at f\"ur Mathematik, 09107 Chemnitz, Germany}

\maketitle
\begin{abstract}
In this note we prove Minami's estimate for a class of discrete alloy-type models with a sign-changing single-site potential of finite support. 
We apply Minami's estimate to prove Poisson statistics for the energy level spacing. 
Our result  is valid for random potentials which are
in a certain sense sufficiently close to the standard Anderson potential (rank one perturbations coupled with i.i.d.\ random variables). 
\end{abstract}
\section{Introduction}
This paper is devoted to the study of spectral properties, in particular eigenvalue statistics, of random Schr\"odinger operators. The specific model we are interested in is the \emph{discrete alloy-type model}, which is defined by the family of Schr\"odinger operators 
\begin{equation} \label{eq:hamiltonian}
 H_\omega = -\Delta + \lambda V_\omega , \quad \omega \in \Omega = \bigtimes_{k \in \mathbb{Z}^d} \mathbb{R}, \quad \lambda > 0 ,
\end{equation}
on $\ell^2 (\mathbb{Z}^d)$. Here, $\Delta$ denotes the discrete Laplacian and $V_\omega$ 
is the multiplication operator by the function
\begin{equation}
 \label{eq:potential}
V_\omega (x) = \sum_{k \in \mathbb{Z}^d} \omega_k u(x-k).  
\end{equation}
It is assumed that the projections $\Omega \ni \omega \mapsto \omega_k$, $k \in \mathbb{Z}^d$, called \emph{coupling constants}, are independent identically distributed (i.i.d.)\ random variables, and that $\omega_0$ is distributed according to some probability measure $\mu$ on $\mathbb R$ with compact support. The function $u : \mathbb{Z}^d \to \mathbb{R}$ is called \emph{single-site potential} and is assumed to be in $\ell^1 (\mathbb{Z}^d ; \mathbb R)$. Studying random Schr\"odinger operators is motivated from solid state physics: 
each configuration $\omega \in \Omega$ corresponds to a possible realization of the random medium, while the product measure on the space $\Omega$ describes the distribution of the individual realizations.
\par
Anderson \cite{Anderson-58} argued, that the solutions of the time-dependent Schr\"{o}dinger equation, the \emph{wave functions}, 
stay localized in space for all time, giving bound states, in certain disorder/energy regimes. 
This phenomenon is called \emph{localization}. 
It is in contrast to the situation encountered with periodic Schr\"odinger operators, 
where the wave functions spread in space  for time tending to infinity (scattering states). 
Mathematically, bound or localized states can be related via the celebrated RAGE-theorem
to intervals $I \subset \mathbb{R}$ in the almost sure spectrum of $H_\omega$, 
where the continuous spectrum of $H_\omega$ is empty for almost all $\omega \in \Omega$. 
There are also definitions of localization from the dynamical point of view, 
see  e.g.\ \cite{delRioJLS-96a} for an early and \cite{Klein-08} for a recent paper.
\par
If localization occurs in some interval $I \subset \mathbb{R}$ the entire spectrum in $I$ corresponds to eigenvalues, 
and their closure. Thus it is natural to ask about the distribution of the eigenvalues. 
Physicists expect that there is no level repulsion of the energy states in the localized regime; 
in fact, that the eigenvalues are distributed independently on the interval $I$, cf.~e.g.~\cite{DisertoriR-08,Efetov-97}.
\par
The first precise result in this direction, namely that the point process associated to rescaled eigenvalues converges to a Poisson process 
has been obtained by S.~Molchanov in \cite{Molchanov-81}.
It concerns a one-dimensional continuum random Schr\"odinger operator on $L^2(\RR)$ 
and relies in its proofs on the papers \cite{GoldsheidMP-77} and \cite{Molchanov-78a}.
An associated central limit theorem was proven in \cite{Reznikova-81a}.
\par
For the purposes of this paper a result \cite{Minami-96} of  Minami 
will be more relevant, where he proved the analogous result for the \emph{i.i.d.\ Anderson model}, 
i.e.\ the operator described in \eqref{eq:hamiltonian} and \eqref{eq:potential}
in the case where $u = \delta_0$ and the measure $\mu$ has a bounded density.
The key feature of Minami's proof is the so-called Minami estimate 
given in  Ineq.~\eqref{eq:minamiiid} below.
\par
Subsequently, Minami's result has stimulated further research in this direction. For example, in \cite{Nakano-06} and \cite{KillipN-07} the authors study the joint distribution of energy levels and localization centers of eigenfunctions. Interestingly, they use the bounds from \cite{Minami-96} as a key tool for their analysis, while in the approach of Molchanov \cite{Molchanov-81, Molchanov-78a} quantitative estimates about the exponential decay of eigenfunctions around localization centers is one step in the proof of asymptotic Poisson statistics of rescaled eigenvalues. 
\par
The papers \cite{GrafV-07,BellissardHS-07} prove Minami's estimate 
for more general background operators than the nearest neighbor Laplacian 
but still for an i.i.d.\ potential, i.e.\ in the case where $u=\delta_0$. In \cite{CombesGK-09} it was shown that one can treat 
random coupling constants $\omega_j$ with H\"older continuous distribution $\mu_j$ as well.
\par
Germinet and Klopp continued the research on spectral statistics 
for random Schr\"odinger operators in \cite{GerminetK-11b}. 
In particular, they give an abstract result that a Wegner estimate and a weakened Minami estimate imply Poisson statistics for a large class of discrete 
random Schr\"odinger operators. 
They also weaken the hypothesis on the existence and positivity of the DOS by 
assuming only a quantified positivity of the density of states measure, 
cf.\ Assumption (Pos) in Section~\ref{sec:poisson}. 
Consistently with this weakened assumption, they scale the eigenvalues 
by the integrated density of states (IDS), cf.\ Eq.~\eqref{eq:scaleev}.
We will apply \cite{GerminetK-11b}
to conclude Poisson statistics from our version of Minami's estimate. 
\par
In \cite{GerminetK-11} an enhanced version 
of a Wegner and Minami estimate is proven, where the interval length $\lvert I \rvert$ 
is replaced by $N(I)$. Moreover, the positivity assumption (Pos) 
is again weakened and does not depend anymore on Minami's estimate. 
However, the results of \cite{GerminetK-11} are stated for the i.i.d.\ Anderson model only, 
and hence cannot be applied to the model we consider here.
In \cite{Klopp-12} a proof of Minami's estimate for one-dimensional random Schr\"odinger operators is established. Roughly speaking, Theorem 1.5 in \cite{Klopp-12} states that a Minami-type estimate holds in any region where localization and a Wegner estimate holds. This Minami-type estimate is then applied to results on eigenvalue statistics for one-dimensional models. 
\par
In \cite{CombesGK-10} important progress was made concerning the proof of
Poisson statistics for alloy-type models on $L^2(\RR^d)$. 
However, according to the information provided by F. Germinet, 
not all steps of the proof are completely correct. 
\par
Our main result is a generalization of Minami's estimate for discrete alloy-type models
with a single-site potential whose support has more than one element. 
Our single-site potential may even change its sign. 
Assumption \ref{ass:circulant} in Section~\ref{sec:model}
specifies the setting in which our result holds.
\par
The main challenge of models of the type \eqref{eq:hamiltonian}
is twofold. First, previous results on Minami's estimate are based on the fact that $u = \delta_0$, 
i.e.\ the perturbation of the discrete Laplacian is of rank one with respect to the random variables. 
Second, our single-site potential $u$ may change its sign which leads 
to negative correlations of the potential. 
This results in a non-monotone dependence on the random variables of certain spectral quantities. 
This lack of monotonicity slowed down progress in the study
on localization  for such families of operators, which is a prerequisite for Poisson statistic. 
Results on localization in the non-monotone case are far more scarce
and were achieved more recently than for the standard i.i.d.\ Anderson model, see e.g.,
\cite{Klopp-95a,Veselic-01,Veselic-02a,HislopK-02,Veselic-10a,Veselic-10b,ElgartTV-10,ElgartTV-11,Krueger-12,CaoE-12,ElgartSS-12}. 
\par
To our knowledge, our result is the first one on  Minami's estimate 
for multi-dimensional random models 
with a single-site potential which is not of rank one. 
In the proof we rely on a certain transformation of the probability space, 
which was used earlier for non-monotone models to obtain 
Wegner estimates and finiteness of Green function fractional moments \cite{Veselic-02a,Veselic-10a,TautenhahnV-10,ElgartTV-11}. 
In retrospect it is not surprising that tools developed for non-monotone models are useful 
in the context of Minami estimates: 
Namely, averaging over local environments turned out 
to be an efficient way to obtain Wegner and Green function fractional moments bounds, 
cf.\ e.g.\ \cite{Veselic-11,ElgartTV-11}. 
On the other hand it is obvious that for Minami's estimate multiple averages have to be performed.
\section{Precise assumptions on the model} 
\label{sec:model}
On the Hilbert space $\ell^2 (\mathbb{Z}^d)$ we consider the discrete alloy-type model, given by the family of discrete Schr\"odinger operators as defined in Eq.~\eqref{eq:hamiltonian}.
Recall that $\omega$ is an element of the probability space $(\Omega , \mathcal{A} , \mathbb{P})$, where $\Omega = \times_{k \in \mathbb{Z}^d} \mathbb{R}$, $\mathcal{A}$ is the $\sigma$-Algebra generated by the cylinder sets and $\mathbb{P}$ is the product measure $\mathbb{P} = \prod_{k \in \mathbb{Z}^d} \mu$. Here $\mu$ denotes an arbitrary probability measure on $\mathbb{R}$ with compact support. 
\par
Next we provide an additional assumption on the type of disorder, i.e., on $u$ and $\mu$, which is required for the main result on Minami's estimate.
\begin{assumptionDis}
\label{ass:circulant}
We assume that $\operatorname{supp}  u$ is compact, its Fourier transform $\hat u \colon [0,2\pi)^d \to
\mathbb{C}$, i.e.,
\[
 \hat u (\theta) = \sum_{k \in \mathbb{Z}^d} u(k) {\rm e}^{\ensuremath{{\mathrm{i}}}k \cdot \theta} ,
\]
does not vanish, and that the measure $\mu$ has a density $\rho \in W^{2,1} (\mathbb{R})$.
\end{assumptionDis}
\begin{remark}
Let us discuss Assumption \ref{ass:circulant} on the single-site potential $u$. 
It is  satisfied if $u$ obeys for some $k\in\mathbb{Z}^d$ the condition
\begin{equation} \label{eq:small}
 \lvert u(k) \rvert > \sum_{j \not = k} \lvert u(j) \rvert ,
\end{equation}
which may be interpreted as a decay condition on the single-site potential. 
Hence our result applies to single-site potentials 
of rank one corresponding to $u(k)$ plus small 
perturbations corresponding to the other values of the single-site potential, see Fig.~\ref{fig:potential}.%
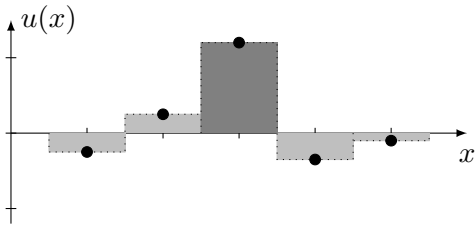
\begin{figure}[th]\centering 
\begin{tikzpicture}
\draw[-latex] (0,0)--(6,0);
\draw[-latex] (0,-1.2)--(0,1.5);
\draw (6,-0.3) node {$x$};
\draw (0.5,1.5) node {$u(x)$};
\foreach \x in {1,2,3,4,5}
 \draw (\x,-2pt)--(\x,2pt);
\foreach \y in {-1,0,1}
 \draw (-2pt,\y)--(2pt,\y);
\usetikzlibrary{patterns}
 \filldraw[black!25] (0.5,0)--(0.5,-0.25)--(1.5,-0.25)--(1.5,0);
 \filldraw[black!25] (1.5,0)--(1.5,0.25)--(2.5,0.25)--(2.5,0);
 \filldraw[black!50] (2.5,0)--(2.5,1.2)--(3.5,1.2)--(3.5,0);
 \filldraw[black!25] (3.5,0)--(3.5,-0.35)--(4.5,-0.35)--(4.5,0);
 \filldraw[black!25] (4.5,0)--(4.5,-0.1)--(5.5,-0.1)--(5.5,0);
\filldraw (1,-0.25) circle(2pt);
\filldraw (2,0.25) circle(2pt);
\filldraw (3,1.2) circle(2pt);
\filldraw (4,-0.35) circle(2pt);
\filldraw (5,-0.1) circle(2pt);
\draw[dotted] (0.5,0)--(0.5,-0.25)--(1.5,-0.25)--(1.5,0.25)--(2.5,0.25)--(2.5,1.2)--(3.5,1.2)--(3.5,-0.35)--(4.5,-0.35)--(4.5,-0.1)--(5.5,-0.1)--(5.5,0);
\draw[dotted] (2.5,0)--(2.5,0.25);
\draw[dotted] (4.5,0)--(4.5,-0.1);
\end{tikzpicture}
\caption{Illustration of condition \eqref{eq:small} on the single-site potential\label{fig:potential}}
\end{figure}
\par
By evaluating $\hat u (0)$ we see that Assumption \ref{ass:circulant} is not satisfied 
if the mean value $\sum_k u(k)$ vanishes. It is known that the analysis of 
discrete and continuum alloy-type models becomes more intricate 
when the mean value of $u$ is zero, cf.\ \cite{Klopp-02c,Veselic-10a}.
\end{remark}
The previous remark can be specified as follows: our results apply to the i.i.d.\ Anderson potential $V_\omega^{\rm A} (x) = \sum_{k \in \ZZ^d} \omega_k \delta_{x-k}$ plus a small perturbation. To be more precise, for any  $v: \ZZ^d \to \RR$ with compact support 
our assumptions are satisfied for the random potential 
\begin{align*}
 V_\omega (x) = V_\omega^{\rm A} (x) + \kappa \sum_{k \in \ZZ^d} \omega_k v(x-k) , 
\end{align*}
if $\kappa > 0$ is sufficiently small (depending on $v$) 
and $\rho \in W^{2,1}$.
In this case one would set $u : \ZZ^d \to \RR, \  u(x) = \delta_x + \kappa v(x) $ and infer that Ineq.~\eqref{eq:small} is satisfied.
\section{Minami's estimate using circulant-matrices} \label{sec:circulant}
Let us recall Minami's estimate for the Anderson model
which is the special case of the discrete alloy-type model where $u = \delta_0$. For this purpose we introduce some notation.
\par
We use the symbol $\mathbb{E}$ to denote the average over the collection of random variables $\omega_k$, $k \in \mathbb{Z}^d$. 
For $\Lambda \subset \mathbb{Z}^d$ we denote by $H_{\omega , \Lambda} : \ell^2 (\Lambda) \to \ell^2 (\Lambda)$ the natural restriction of $H_\omega$ to the set $\Lambda$. For $x,y \in \Lambda$ and $z \in \mathbb{C} \setminus \sigma (H_{\omega , \Lambda})$ we set $G_{\omega , \Lambda} (z;x,y) = \langle \delta_x , (H_{\omega, \Lambda} - z)^{-1} \delta_y \rangle$. Here $\delta_k \in \ell^2 (\Lambda)$ denotes the Dirac delta function. For $x \in \mathbb{Z}^d$ and $L>0$ we use the notation $\Lambda_{L,x} = \{y \in \mathbb{Z}^d \colon \lvert y-x \rvert_{\infty} \leq L\}$ and $\Lambda_L = \Lambda_{L,0}$. 
\par
Minami showed in \cite{Minami-96} for $u = \delta_0$ and in the case where $\mu$ has a density $\rho \in L^{\infty} (\mathbb{R})$ that for all $\Lambda \subset \mathbb{Z}^d$ finite and all $x,y \in \Lambda$ one has
\begin{equation} \label{eq:minamiiid}
 \mathbb{E} \left( \det \left\{ \Im \begin{pmatrix}
  G_{\omega , \Lambda} (z;x,x) & G_{\omega , \Lambda} (z;x,y) \\
  G_{\omega , \Lambda} (z;y,x) & G_{\omega , \Lambda} (z;y,y)
 \end{pmatrix}  \right\} \right) \leq \pi^2 \lVert \rho \rVert_\infty^2 .
\end{equation}
The proof essentially makes use of two ingredients. The first one is the so-called Krein formula, a representation of the Green function. The second one is the following lemma, which was proven in \cite{Minami-96} for symmetric matrices $M$ with $\Im M > 0$ and generalized in \cite{GrafV-07,BellissardHS-07} to arbitrary matrices $M$ with positive imaginary part in order to consider more general background operators. 
These tools will be relevant also in our proof.
\begin{lemma}\label{lemma:graf}
 Let $M = (m_{ij})_{i,j=1}^2$ with $\Im M > 0$. Then
\[
 \int_\mathbb{R} \int_\mathbb{R} \det \left( \Im \left[\begin{pmatrix}
                                 v_1 & 0 \\ 0 & v_2
                                \end{pmatrix} - M
 \right]^{-1} \right) \mathrm{d}v_1 \mathrm{d}v_2 \leq \pi^2 .
\]
\end{lemma}
In the remainder of this section we prove Minami's estimate for more general single-site potentials, i.e.\ under Assumption \ref{ass:circulant}. 
The idea is to use a special transformation of the probability space. This approach has been used earlier, 
e.g., to prove Wegner estimates for Anderson models on the lattice and in the continuum with a sign-changing single-site potential, see e.g.\ \cite{Veselic-01,Veselic-02a,TautenhahnV-10,Veselic-10b}.
\par
First we introduce the mentioned linear transformation. 
For $\Lambda \subset \mathbb{Z}^d$ finite and $\operatorname{supp} u$ compact we define $l = \{\max \{ \lVert x \rVert_\infty \colon x \in \Lambda \}\}$, choose $R \in \NN$ such that $\operatorname{supp} u \subset \Lambda_R$ and set $L = l+R$. Then the cube $\Lambda^+ := \Lambda_{l+R} = \Lambda_L$ contains the set
$\bigcup_{x \in \Lambda} \{k \in \mathbb{Z}^d \colon
u(x-k) \not = 0 \}$, which is the set of all lattice sites $j$ whose coupling constant influences the potential values in $\Lambda$.
Therefore, the restricted operator $H_{\omega ,\Lambda}$ depends only on the random variables $\omega_k$, $k \in \Lambda^+$. This means, if we consider the expectation of some measurable function of $H_{\omega , \Lambda}$ it suffices to average merely with respect to the finite collection of random variables $\omega_k$, $k \in \Lambda^+$.
\par
Let now $A:\ell^1 (\mathbb{Z}^d) \to \ell^1 (\mathbb{Z}^d)$ be the linear operator whose coefficients in the canonical orthonormal basis are given by $A(j,k) = u(j-k)$ for $j,k \in \mathbb{Z}^d$. Since $u$ has compact support, the operator $A$ is bounded. 
If $\hat u$ does not vanish (as required by Assumption \ref{ass:circulant}), the operator $A$ has a bounded inverse by the so-called $1/f$-Theorem of Wiener and we have
\begin{equation} \label{eq:Cu}
C_u := \lVert A^{-1} \rVert_1 < \infty ,
\end{equation}
see \cite{Veselic-10b} for details. Moreover, there exists an invertible matrix $A_{\Lambda^+} : \ell^1 (\Lambda^+) \to \ell^1 (\Lambda^+)$ satisfying 
\begin{equation} \label{eq:circ}
 A_{\Lambda^+} (i,j) = u (i-j) \quad \text{for all $i\in \Lambda$ and $j \in
\Lambda^+$}, \quad \text{and} \quad \lVert A_{\Lambda^+}^{-1} \rVert_1 \leq C_u ,
\end{equation}
One possible choice of $A_{\Lambda^+} : \ell^1 (\Lambda^+) \to \ell^1 (\Lambda^+)$ is 
\begin{equation*} \label{eq:matrix_circ}
A_{\Lambda^+} (i,j) = u\bigl( \pi_{L}(i-j) \bigr) , \quad i,j \in \Lambda^+ .
\end{equation*}
Here the map $\pi_{L} : \ZZ^d \to \Lambda^+$ is defined by $\pi_{L} (x) = \Lambda^+ \cap ((2L+1)\ZZ^d + x)$. Note that (due to the definition of the projection map $\pi_L$) this choice of $A_{\Lambda^+}$ is a multi-dimensional circulant matrix. That the first condition in \eqref{eq:circ} is satisfied for this circulant matrix follows immediately by the choice of $R$ and the definition of $\pi_L$. The invertibility of $A_{\Lambda^+}$ and the second property of \eqref{eq:circ} follows by a calculation similar to the proof of Proposition 9 in \cite{Veselic-10b}.
\par
In the special case $d = 1$, $\Lambda = \Lambda_1$ and $\operatorname{supp} u = \{ -1,0,1 \}$, we have $l = 1$, can choose $R = 1$, and the matrix from Eq.~\eqref{eq:matrix_circ} reads in the canonical basis
\[
A_{\Lambda^+} = \begin{pmatrix}
u(0) & u(-1) &0 &0 & u(1)  \\
u(1) & u(0) & u(-1) &0 &0  \\
0    & u(1) & u(0) & u(-1) &0  \\
0    &0 & u(1) & u(0) & u(-1) \\
u(-1)    &0 &0 & u(1) & u(0)  \\
                \end{pmatrix} .
\]
\par
Set $\omega_{\Lambda^+} = (\omega_k)_{k \in \Lambda^+}$ and define the new (random) vector $\zeta = (\zeta_k)_{k \in \Lambda^+} = A_{\Lambda^+} \omega_{\Lambda^+}$. Then, by Eq.~\eqref{eq:circ}, the vector $\zeta$ satisfies $\zeta_k = V_\omega (k)$ for all $k \in \Lambda$ (but not necessarily for $k \in \Lambda^+ \setminus \Lambda$). 
This property is important for the proof of our main technical result:
\begin{theorem}\label{thm:result1}
Let $\Lambda \subset \mathbb{Z}^d$ be finite and Assumption \ref{ass:circulant} satisfied. 
Then we have for all $x,y \in \Lambda$ with $x \not = y$, all $z \in \mathbb{C}$ with $\Im z > 0$ and all $\lambda > 0$
\[
 \mathbb{E} \left( \det \left\{ \Im \begin{pmatrix}
  G_{\omega , \Lambda} (z;x,x) & G_{\omega , \Lambda} (z;x,y) \\
  G_{\omega , \Lambda} (z;y,x) & G_{\omega , \Lambda} (z;y,y)
 \end{pmatrix}  \right\} \right) \leq \left(\frac{\pi}{\lambda}\right)^2 C_{\rm Min} ,
\]
where 
\[
 C_{\rm Min} = \frac{C_u^2}{4} \max\{\lVert \rho' \rVert_1^2 , \lVert \rho'' \rVert_1\} 
\]
and $C_u$ is the constant from Eq.~\eqref{eq:Cu}.
\end{theorem}
As in \cite{GrafV-07,BellissardHS-07,CombesGK-09} the theorem still holds if 
the non-random part of $H_\omega$ is not the negative Laplacian $-\Delta$, 
but rather an arbitrary self-adjoint operator $H_0$.
\begin{proof}
 Let $\Lambda^+$, $A_{\Lambda^+}$ be as above and set $B :=  A_{\Lambda^+}^{-1}$. For $x,y \in \Lambda$ with $x \not = y$ the second resolvent identity implies the so-called Krein formula, see e.g.\ \cite{AizenmanM-93}, i.e.
\[
g(z) = \begin{pmatrix}
  G_{\omega , \Lambda} (z;x,x) & G_{\omega , \Lambda} (z;x,y) \\
  G_{\omega , \Lambda} (z;y,x) & G_{\omega , \Lambda} (z;y,y)
 \end{pmatrix} 
= \lambda^{-1} \left( \begin{pmatrix}
              V_\omega (x) & 0 \\ 0 & V_\omega (y) 
             \end{pmatrix} - M
 \right)^{-1} , 
\]
where $M$ is the $2\times 2$ matrix
\[
M= - \lambda^{-1} \begin{pmatrix}
    \hat G_{\omega , \Lambda} (z;x,x) & \hat G_{\omega , \Lambda} (z;x,y) \\
    \hat G_{\omega , \Lambda} (z;y,x) & \hat G_{\omega , \Lambda} (z;y,y)
   \end{pmatrix}^{-1} .
\]
Here $\hat G_{\omega , \Lambda} (z;u,v)$ is the Green function of the operator $\hat H_{\omega , \Lambda} = H_{\omega , \Lambda} - V_\omega (x) P_x - V_{\omega} (y) P_y$, where, for $k \in \Lambda$, $P_k : \ell^2 (\Lambda) \to \ell^2 (\Lambda)$ is the orthogonal projection on the state $\delta_k$, i.e.\ $P_k \psi = \psi (k) \delta_k$. Let us also note that $\Im M$ is positive definite if $\Im z > 0$, see e.g.\ \cite{Graf-94,GrafV-07}, and therefore $\operatorname{diag} ( V_\omega(x) , V_\omega(y)) - M$ is invertible.
\par
Set $\omega_{\Lambda^+} = (\omega_k)_{k \in
\Lambda^+}$, $k(\omega_{\Lambda^+}) = \prod_{k \in \Lambda^+} \rho
(\omega_k)$, $\ensuremath{\mathrm{d}} \omega_{\Lambda^+} = \prod_{k \in \Lambda^+} \ensuremath{\mathrm{d}} \omega_k$
and $n = \lvert \Lambda^+ \rvert$. Using the substitution
$\zeta = (\zeta_k)_{k \in \Lambda^+} = A_{\Lambda^+} \omega_{\Lambda^+}$ 
we obtain from Eq.~\eqref{eq:circ}
\begin{align}
 \mathbb{E} \left( \det [\Im g(z)] \right) & = \int_{\mathbb{R}^n} \!\!\!\!\! \det [\Im g(z)] k(\omega_{\Lambda^+}) \mathrm{d} \omega_{\Lambda^+} 
\nonumber \\[1ex] 
&=\lambda^{-2}\int_{\mathbb{R}^n} \!\!\!\!\! \det \left( \Im \left(\begin{pmatrix}
              \zeta_x & 0 \\ 0 & \zeta_y 
             \end{pmatrix} -M
 \right)^{-1} \right) \tilde k(\zeta)  \mathrm{d} \zeta \label{eq:desintegration}
\end{align}
where $\tilde k (\zeta) = k(B \zeta) \lvert \det B \rvert$ and $\mathrm{d} \zeta = \prod_{k \in \Lambda^+} \mathrm{d} \zeta_k$.
Since $\zeta_k = V_\omega (k)$ for all $k \in \Lambda$ by construction, the matrix $M$ does not depend on the parameters $\zeta_x$ and $\zeta_y$, though it may be correlated with the random variables $\zeta_x$ or $\zeta_y$. This dependence is encoded in the joint density $\tilde k(\zeta)$ of the random variables $\zeta_k$, $k \in \Lambda^+$. By Fubini's theorem and Lemma~\ref{lemma:graf} we obtain
\begin{equation} \label{eq:supsup}
 \mathbb{E} \left( \det [\Im g(z)] \right) \leq \left(\frac{\pi}{\lambda}\right)^2 \int_{\mathbb{R}^{n - 2}} \left(\sup_{\zeta_y} \sup_{\zeta_x} k(B \zeta)\right)  \lvert \det B \rvert \!\!\!\! \prod_{k \in \Lambda^+ \setminus \{x,y\}} \!\!\!\! \mathrm{d} \zeta_{k} .
\end{equation}
Now we use a special case of the Sobolev imbedding theorem \cite[Theorem 4.12]{Adams-75}, namely that $W^{2,1} (\mathbb{R}^2) \hookrightarrow L^{\infty} (\mathbb{R}^2)$. In particular, by Lemma~\ref{lemma:sobolev} we have
\begin{equation} \label{eq:viertel}
\lVert  f \rVert_\infty \leq \frac{1}{4}  \lVert D^{(1,1)} f \rVert_{1} \quad \text{for $f \in W^{2,1} (\mathbb{R}^2)$} .
\end{equation}
Note that for fixed $\zeta_k$, $k \in \Lambda^+ \setminus \{x,y\}$ the mapping 
\[
 (\zeta_x , \zeta_y) \mapsto \tilde k (\zeta)
\]
is an element of $W^{2,1} (\mathbb{R}^2)$ since $\rho \in W^{2,1} (\mathbb{R})$ by assumption. For the weak derivative $\partial_{\zeta_x} \partial_{\zeta_y} \tilde k$ we calculate 
\begin{align}
 \partial_{\zeta_x} \partial_{\zeta_y} \tilde k (\zeta) &= \sum_{j \in \Lambda^+} B_{jy} \Biggl( \rho'' \left( (B \zeta)_j \right) B_{jx} \prod_{\genfrac{}{}{0pt}{2}{k \in \Lambda^+}{k \not = j} } \rho \left( (B \zeta)_k \right)   \Biggr. \nonumber \\
&\quad + \Biggl.\sum_{\genfrac{}{}{0pt}{2}{l \in \Lambda^+}{l \not = j} } B_{lx} \rho' \left((B\zeta)_j \right) \rho' \left( (B\zeta)_l \right)  \prod_{\genfrac{}{}{0pt}{2}{k \in \Lambda^+}{k \not= j,l} } \rho \left( (B\zeta)_k \right) \Biggr)  . \label{eq:derivative}
\end{align}
Hence we obtain from Ineq.~\eqref{eq:supsup} and Ineq.~\eqref{eq:viertel}
\begin{align} \label{eq:sob2}
 \mathbb{E} \left( \det [\Im g(z)] \right) &\leq \frac{\pi^2}{4\lambda^2}  \int_{\mathbb{R}^{n}} \lvert \partial_{\zeta_x} \partial_{\zeta_y} \tilde k (\zeta) \rvert \lvert \det B \rvert \!\!\!\! \prod_{k \in \Lambda^+ \setminus \{x,y\}} \!\!\!\! \mathrm{d} \zeta_{k} .
\end{align}
When substituting back into original coordinates, $\omega_{\Lambda^+} = B \zeta$, we obtain from Eq.~\eqref{eq:derivative}, Ineq.~\eqref{eq:sob2} and the triangle inequality
\begin{align}
\label{eq:final}
\mathbb{E} \left( \det [\Im g(z)] \right) &\leq \frac{\pi^2}{4\lambda^2} \sum_{j \in \Lambda^+} \lvert B_{jy} \rvert \Bigl[ \lVert \rho'' \rVert_1 \lvert B_{jx} \rvert + \sum_{\genfrac{}{}{0pt}{2}{l \in \Lambda^+}{l \not = j}} \lVert \rho' \rVert_1^2 \lvert B_{lx} \rvert \Bigr] \\
& \leq \frac{\pi^2}{4 \lambda^2} \max \{\lVert \rho' \rVert_1^2 , \lVert \rho'' \rVert_1 \} \lVert B \rVert_1^2  . \nonumber
%
\end{align}
We use $\lVert B \rVert_1^2 \leq C_u^2$ as discussed in Ineq.~\eqref{eq:circ} and obtain the statement of the theorem.
\end{proof}
\begin{remark}
Another way to apply Lemma~\ref{lemma:graf} to Eq.~\eqref{eq:desintegration} is to use conditional densities:
\begin{multline*}
 \mathbb{E} \left( \det [\Im g(z)] \right) =
\\ \lambda^{-2}\int_{\mathbb{R}^{n-2}} \!\!\!\!\!\! F(\zeta^\perp)
\int_{\mathbb{R}^{2}} \!\!\!  \det \left( \Im \left(\begin{pmatrix}
              \zeta_x & 0 \\ 0 & \zeta_y 
             \end{pmatrix} -M
 \right)^{-1} \right) f(\zeta_x,\zeta_y;\zeta^\perp) \mathrm{d} \zeta_x \mathrm{d} \zeta_y \,  \mathrm{d} \zeta^\perp,
\end{multline*}
where $\zeta^\perp:= (\zeta_k)_{k \in \Lambda^+\setminus\{x,y\}} $, denotes the parameters except those associated with $x$ and $y$,
$F(\zeta^\perp):= \int \mathrm{d} \zeta_x \int \mathrm{d} \zeta_y  \tilde k(\zeta) $ is the marginal and 
$f(\zeta_x,\zeta_y;\zeta^\perp):= \tilde k(\zeta)/ F(\zeta^\perp) $ the conditional density.
The mentioned lemma gives
\begin{equation*}
 \mathbb{E} \left( \det [\Im g(z)] \right) 
\leq   \frac{\pi^2}{\lambda^2}  C_{\rm Min} \int_{\mathbb{R}^{n-2}}  F(\zeta^\perp) \sup_{\zeta_x}\sup_{\zeta_y} f(\zeta_x,\zeta_y;\zeta^\perp) \mathrm{d} \zeta^\perp .
\end{equation*}
Now it's tempting to estimate and pull the conditional density  out as $\sup_{\zeta_x} \sup_{\zeta_y} \sup_{\zeta^\perp} f(\zeta_x,\zeta_y;\zeta^\perp) $.
However such suprema are typically infinite, see e.g. the discussion in \cite{TautenhahnV-10a}. 
One can try to estimate the averaged quantity $\int_{\mathbb{R}^{n-2}}  F(\zeta^\perp) \sup_{\zeta_x}\sup_{\zeta_y} f(\zeta_x,\zeta_y;\zeta^\perp) \mathrm{d} \zeta^\perp $.
This actually leads to estimates similar in nature to the one performed in Ineq.~\eqref{eq:supsup} to \eqref{eq:final}.
\end{remark}
Theorem~\ref{thm:result1} has an important corollary, a bound on the probability of finding at least two eigenvalues of $H_{\omega , \Lambda}$ in a certain interval. 
\begin{corollary} \label{cor:minami}
 Let Assumption \ref{ass:circulant} be satisfied, $\Lambda \subset \mathbb{Z}^d$ finite and $I \subset \mathbb{R}$ be a bounded interval. Then we have for all $\lambda > 0$
\begin{align}
\label{eq:Cebysev}
 \mathbb{P} \bigl\{ \operatorname{Tr} \chi_{I} (H_{\omega , \Lambda}) \geq 2 \bigr\} 
&\leq 
\frac{1}{2}\mathbb{E} \bigl( (\operatorname{Tr} \chi_{I} (H_{\omega , \Lambda}))^2  - \operatorname{Tr} \chi_{I} (H_{\omega , \Lambda}) \bigr) 
\\ \nonumber
&\leq 
\frac{1}{2}\left(\frac{\pi}{\lambda}\right)^2 C_{\rm Min} \lvert I \rvert^2 \lvert \Lambda \rvert^2 ,
\end{align}
where $C_{\rm Min}$ is the constant from Theorem~\ref{thm:result1}.
\end{corollary}
The proof of this corollary is due to \cite{Minami-96}, see also \cite[Appendix]{KleinM-06} or \cite{CombesGK-09}. 
These papers study the case $u = \delta_0$, however the proof applies 
to the discrete alloy-type model studied in this note, as well. 
\section{Application to Poisson statistics for energy level spacing} \label{sec:poisson}
In this section we prove under Assumption~\ref{ass:circulant} that near an energy where Anderson localization holds
and where the IDS increases, there is no correlation between eigenvalues of $H_{\omega , \Lambda}$ if $\lvert \Lambda \rvert$ is large. 
To characterize Anderson localization we will use here the following finite volume criterion.
\begin{assumption_FVC} \label{ass:FVC} 
Let $I \subset \mathbb{R}$. We say that Assumption (FVC) is satisfied in $I$ if for all $E \in I$ there exists $\Theta > 3d-1$ such that
\[
 \limsup_{L\to \infty} \mathbb{P} \left\{ \forall x,y \in \Lambda_{L,0} , \ \lvert x-y \rvert_\infty \geq \frac{L}{2} : \lvert  G_{\omega , \Lambda_{L,0}} (E;x,y) \rvert \leq L^{-\Theta} \right\} = 1 .
\]
\end{assumption_FVC}
As we will discuss in Section~\ref{s:Assumptions} this finite volume criterion is indeed equivalent to various notions of localization,
once an (sufficiently good) Wegner estimate is available.
\par
Let $L \in \mathbb{N}$ and and $E_1^\omega (\Lambda_L) \leq E_2^\omega (\Lambda_L) \leq \ldots \leq E_{\lvert \Lambda_L \rvert}^\omega (\Lambda_L)$ be the eigenvalues of $H_{\omega , \Lambda_L}$ repeated according to multiplicity. Since $(H_\omega)_\omega$ is an ergodic family of random operators, the IDS exists as a (non-random) distribution function $N :  \mathbb{R} \to [0,1]$, satisfying for almost all $\omega \in \Omega$
\[
 N(E) = \lim_{L \to \infty} \frac{1}{\lvert \Lambda_L \rvert} \# \{ j \in \mathbb{N} \colon E_j^\omega (\Lambda_L) \leq E \},
\]
at all continuity points of $N$. In particular, if Assumption \ref{ass:circulant} is satisfied, the IDS is known to be Lipschitz continuous \cite{Veselic-10b}. 
Let us now introduce a second hypothesis which may be interpreted as a quantitative growth condition on the IDS or a positivity assumption on the density of states measure.
\begin{assumption_pos} 
 Let $E_0 \in \mathbb R$ and $\kappa \geq 0$. We say that Assumption (Pos) is satisfied for $E_0$ and $\kappa $ if for all $a<b$ there exists $C,\epsilon_0 > 0$ such that for all $\epsilon \in (0,\epsilon_0)$ there holds
\[
 \lvert N(E_0 + a \epsilon) - N(E_0 + b \epsilon) \rvert \geq C \epsilon^{1 + \kappa } .
\]
\end{assumption_pos}
For $E_0 \in \mathbb{R}$ we consider the rescaled spectrum $\xi^\omega = (\xi_j^\omega)_{j=1}^{\lvert \Lambda_L \rvert}$, defined by
\begin{equation} \label{eq:scaleev}
 \xi_j^\omega = \xi_j^\omega (L , E_0) = \lvert \Lambda_L \rvert \left(N(E_j^\omega (\Lambda_L)) - N(E_0)\right), \quad j=1,\ldots , \lvert \Lambda_L \rvert ,
\end{equation}
and the associated point process $\Xi : \Omega \to \mathcal{M}_{\rm p}$ given by
\begin{equation} \label{eq:process}
 \Xi^\omega = \Xi^\omega_{L , E_0}  = \sum_{j=1}^{\lvert \Lambda_L \rvert} \delta_{\xi_j^\omega} ,
\end{equation}
where $\delta_x$ is the Dirac measure concentrated at $x$ and $\mathcal{M}_{\rm p}$ is the set of all integer valued Radon measures on $\mathbb{R}$. 
A point process $\Upsilon$ is called Poisson point process with intensity measure $\mu$ if 
\[
 \mathbb{P} \bigl\{ \omega \in \Omega \colon \Upsilon^\omega (A) = k \bigr\} = \mathrm{e}^{-\mu(A)} \frac{\mu (A)^k}{k!}, \quad k=1,2,\ldots 
\]
holds for each bounded Borel set $A \in \mathcal{B} (\mathbb{R})$
and for $A_1 , \ldots , A_n$ disjoint, $\Upsilon (A_1), \ldots , \Upsilon (A_n)$ are independent random variables.
\par 
Let $\Upsilon_n : \Omega \to \mathcal M_{\rm p}$, $n \in \mathbb N$, be a sequence of point processes defined on a probability space $(\Omega , \mathcal A , \mathbb P)$. This sequence is said to converge weakly to a point process $\Upsilon : \tilde \Omega \to \mathcal M_{\rm p}$ defined on a probability space $(\tilde \Omega , \tilde{\mathcal{A}} , \tilde{\mathbb{P}})$, if and only if for any bounded continuous function $\phi:\mathcal M_{\rm p} \to \mathbb R$ there holds
\[
 \lim_{n \to \infty} \int_\Omega \phi (\Upsilon_n^\omega) \mathbb{P} (\mathrm{d} \omega) = \int_{\tilde \Omega} \phi (\Upsilon^{\tilde{\omega}}) \tilde{\mathbb{P}} (\mathrm{d} \omega) .
\]
Let $\Sigma$ denote the almost sure spectrum of the (ergodic) family of operators $H_\omega$, $\omega\in \Omega$. Our main result of this section is
\begin{theorem} \label{theorem:poisson}
 Let Assumption \ref{ass:circulant} be satisfied, $I \subset \Sigma$ be a bounded interval and $E_0 \in I$. Assume that Assumption (FVC) is satisfied in $I$ and Assumption (Pos) is satisfied for $E_0$ and some $\kappa \in [0,1/(1+d))$. 

Then the point process $\Xi$, defined in Eq.~\eqref{eq:process}, converges for $L \to \infty$ weakly to a Poisson process on $\mathbb{R}$ with
Lebesgue measure as the intensity measure.
\end{theorem}
As discussed in Section~\ref{s:Assumptions}, see e.g.\ \cite{ElgartSS-12} for a specific result, Assumption (FVC) is satisfied in $\RR$ if the disorder parameter $\lambda$ in the model \eqref{eq:hamiltonian} is sufficiently large. Thus we obtain the following corollary from Theorem~\ref{theorem:poisson}.
\begin{corollary}
Let Assumption \ref{ass:circulant} be satisfied. Then there exists $\lambda_0 < \infty$ such that for $\lambda \geq \lambda_0$ and for all $E_0 \in \Sigma$ where Assumption (Pos) is satisfied for some $\kappa \in [0,1/(1+d))$, the point process $\Xi$ defined in Eq.~\eqref{eq:process}, converges weakly for $L \to \infty$ to a Poisson process on $\mathbb{R}$ with
Lebesgue measure as the intensity measure.
\end{corollary}
Theorem~\ref{theorem:poisson} will follow from the Minami estimate provided in Section~\ref{sec:circulant}, the Wegner estimate of \cite{Veselic-10b}, 
and the abstract result of  \cite{GerminetK-11b}. Thus, for the proof of Theorem~\ref{theorem:poisson} it will be necessary to 
recall, in an adapted form, Theorem~1.9 of \cite{GerminetK-11b}. It is stated there for an abstract class of random Hamiltonians
 on the lattice including the discrete alloy-type model
studied in this note as a special case. In order to do so, we introduce the following assumptions. 
\begin{assumption_w} \label{ass:wegner}
 Let $I \subset \mathbb{R}$. We say that Assumption (W) is satisfied in $I$ if there exists a constant $C_{\rm W} > 0$ such that for any bounded interval $J \subset I$ and any $L \in \mathbb{N}$ one has
\[
 \mathbb{E} \bigl( \operatorname{Tr} \chi_{J} (H_{\omega , \Lambda_L}) \bigr) \leq C_{\rm W} \lvert J \rvert \lvert \Lambda_L \rvert .
\]
\end{assumption_w}
\begin{assumption_m} \label{ass:minami}
Let $I \subset \mathbb{R}$. We say that Assumption (M) is satisfied in $I$ if there exist constants $C_{\rm Min},\beta > 0$ such that for any bounded interval $J \subset I$ and any $L \in \mathbb{N}$ one has
\[
\mathbb{E} \bigl( (\operatorname{Tr} \chi_{J} (H_{\omega , \Lambda_L}))^2  - \operatorname{Tr} \chi_{J} (H_{\omega , \Lambda_L}) \bigr) 
\leq C_{\rm Min} (\lvert J \rvert \, \lvert \Lambda_L \rvert)^{1+\beta} .
\] 
\end{assumption_m}
\begin{assumption_iad} \label{ass:iad}
 Let $\Lambda , \Lambda' \subset \mathbb{Z}^d$ finite. There exists $R_0 > 0$ such that if $\operatorname{dist}(\Lambda , \Lambda') > R_0$, 
then the random Hamiltonians $H_{\omega , \Lambda}$ and $H_{\omega , \Lambda'}$  are independent random variables. 
\end{assumption_iad}
Now we are in the position to formulate the criterion of \cite{GerminetK-11b}:
\begin{theorem}\label{theorem:gk}
Let $I \subset \Sigma$ be a bounded interval and assume that (IAD), (W), (M) and (FVC) are satisfied in $I$. Pick some $\kappa $ such that 
\[
 0 \leq \kappa < \frac{\beta}{1 + d \beta}
\]
where $\beta$ is defined by (M). Pick $E_0 \in I$ and assume that Assumption (Pos) is satisfied for $E_0$ and $\kappa $.

Then the point process $\Xi$ defined in Eq.~\eqref{eq:process} converges weakly, as $L \to \infty$, 
to a Poisson process on $\mathbb{R}$ with intensity measure equal to the Lebesgue measure.
\end{theorem}
\begin{proof}[Proof of Theorem~\ref{theorem:poisson}]
We have only to verify the hypothesis of Theorem~\ref{theorem:gk}. Assumption (IAD) is satisfied since $u$ has bounded support. A Wegner estimate (W) was proven for $I = \mathbb{R}$ under Assumption \ref{ass:circulant} in \cite{Veselic-10b}. 
Minami's estimate (M) holds for $I = \mathbb{R}$ and $\beta = 1$ due to Corollary~\ref{cor:minami}.
\end{proof}
\begin{remark}
Assuming (W), (M), (FVC) and (IAD) for a general ergodic Schr\"odinger operator on $\ell^2(\ZZ^d)$, 
the paper \cite{GerminetK-11b} presents a plethora of more precise results on the rescaled eigenvalue statistics and even on the joint statistics of rescaled eigenvalues and localization centers.
Since we have all these assumptions for the discrete alloy-type model given by Eq.~\eqref{eq:hamiltonian}, all these abstract results apply.
\par
For the standard Anderson model related results have been obtained 
already in \cite{Nakano-07}.
\end{remark}
\section{Discussion of the Assumptions (FVC) and (Pos)}
\label{s:Assumptions}
First we discuss Assumption (FVC) in some detail. This is necessary,
since our assumptions cover some non-monotone models as well ($u$ may change its sign), 
where a satisfactory understanding of localization poses certain challenges not encountered
in the standard Anderson model with $u = \delta_0$. 
Roughly speaking, Assumption (FVC) for a certain interval $I \subset \mathbb{R}$ corresponds
to the fact that Anderson localization holds in a dynamical sense in the interval $I$. 
More precisely, if the Wegner estimate (W) holds in $I \subset \mathbb R$, 
then (FVC) is equivalent to certain dynamical localization properties
which ensure that the solutions of the Schr\"odinger equation stay trapped in a finite region of space for all time almost surely,
see \cite{GerminetK-04} or \cite[Theorem~6.1]{GerminetK-11b}. 
In particular, these localization properties imply that the continuous spectrum of $H_\omega$ is empty for $\mathbb P$-almost all $\omega \in \Omega$. 
Note that a Wegner estimate (W) is available for the discrete alloy-type model studied in this note once  Assumption \ref{ass:circulant} is satisfied, see \cite{Veselic-10b}.
\par
In the multidimensional setting there are two methods available to prove Anderson localization and hence to verify Assumption (FVC); the multiscale analysis \cite{FroehlichS-83,FroehlichMSS-85} and the fractional moment method \cite{AizenmanM-93,Aizenman-94,Graf-94}. 
\par
In the setting of the multiscale analysis Assumption (FVC) corresponds to the so-called initial scale estimate. 
Roughly speaking, the multiscale analysis implies localization in any energy region $I \subset \mathbb{R}$ where a Wegner estimate and the initial length scale estimate hold. 
Note that the initial length scale estimate may be deduced 
from Lifshitz tails, or 
from a Wegner estimate in the case of sufficiently large disorder \cite{Veselic-10a}. The latter fact is specific for random Schr\"odinger operators on the lattice.
Let us name a few results where the multiscale analysis was applied in certain energy/disorder regimes for a localization proof for the discrete alloy-type model. 
With a certain energy/disorder regime we mean either $I = \mathbb R$ in the case of sufficiently large disorder (large disorder) 
or intervals around edges of the almost sure spectrum for arbitrary disorder $\lambda > 0$ (spectral extrema).
\begin{description}
 \item[\cite{FroehlichS-83,FroehlichMSS-85,DreifusK-89}] consider the case $u = \delta_0$ and $\rho$ bounded and compactly supported (large disorder and spectral extrema).
 \item[\cite{KirschSS-98b}] consider non-negative $u$ with $u(x)  \leq C \lvert x \rvert^{-m}$ for some $m>4d$. 
They consider the continuous alloy-type model, but the results can transferred to the discrete alloy-type model with non negative $u$, cf.~\cite{LeonhardtPTV}, (spectral extrema).
\end{description}
The model studied in this note allows the single-site potential to change its sign, while the above results apply to non-negative single site potentials only. 
If the single site potential changes its sign, certain properties of $H_\omega$ depend in a non-monotonic way on the random parameters. For $u$ with changing sign much less is known. 
Most proofs of the Wegner estimate and the initial length scale estimate for non-negative $u$ do not apply in the general setting. Let us list some results for sign-changing single-site potentials.
\begin{description}
\item[\cite{Klopp-95a,HislopK-02}] prove a Wegner estimate for the continuous alloy-type model near the minimum of the spectrum  
and show that if the IDS exhibits a Lifshitz-tail behavior one can conclude localization (spectral extrema).
 \item[\cite{Veselic-01,Veselic-02a}]  prove a Wegner estimate for the continuous and discrete alloy-type model under condition \eqref{eq:small} 
    and an initial length scale estimate, roughly speaking, if $u$ has only a small negative part (spectral extrema).
\item[\cite{Klopp-02c}] proves localization in the weak disorder regime for the continuum alloy type model if the mean value of the single site potential does not vanish
(spectral extrema).
\item[\cite{Veselic-10a,Veselic-10b}] prove a Wegner estimate for the discrete and the continuous alloy-type model under various conditions on the alloy-type potential
\item[\cite{KloppN-10}] prove Lifshitz tails for generalized alloy-type models in the continuum, which imply the finite volume criterion (FVC). 
 This is not applicable for the lattice models we are considering here.
\item[\cite{Krueger-12}] proves localization using an enhanced version of the multiscale analysis \`{a} la Bourgain for
the discrete alloy-type model  with single site potentials $u$ with exponential decay at strong disorder (large disorder).  

\item[\cite{CaoE-12}] prove localization in the lattice case when $d=3$, $u$ is of compact support, and the density $\rho$ 
is bounded compactly supported and H\"older continuous (spectral extrema). Moreover, they provide quantitative bounds on the localization regimes.
\item[\cite{PeyerimhoffTV-11,LeonhardtPTV}] consider (possibly sign-changing) single-site potentials $u$ with exponential decay. 
The first paper establishes Wegner estimates, the second a finite volume criterion both for the discrete and continuum alloy type model.
\end{description}
\par
Let us now turn to a discussion of results via the fractional moment method which can be used to verify Assumption (FVC). The typical output of the fractional moment method is the so-called fractional moment bound (FMB) as formulated in Ineq.~\eqref{eq:fmb}. Note that the the (FMB) for $x=y$ implies a Wegner estimate, see \cite{ElgartTV-10}. The next lemma shows that the (FMB) also implies Assumption (FVC). 
\begin{lemma}[(FMB) implies (FVC)]
Let $I \subset \mathbb{R}$ and assume that there exist constants $s\in (0,1) $ and $A, \gamma \in (0,\infty)$ 
such that for all $\Gamma \subset \mathbb{Z}^d$, $E \in I$, $\epsilon > 0$ and $x,y \in \Gamma$ we have
\begin{equation} \label{eq:fmb} 
\mathbb{E} \bigl\{\lvert G_{\omega , \Gamma} (E + \mathrm{i} \epsilon;x,y)\rvert^{s} \bigr\}\leq A \mathrm{e}^{-\gamma|x-y|_\infty} .
\end{equation}
Then Assumption (FVC) is satisfied in $I$.
\end{lemma}
\begin{proof}
 The result follows from an application of the sub-additivity of the probability measure and Chebyshev's inequality.
\end{proof}
The fractional moment bound as described in Ineq.~\eqref{eq:fmb} has been verified for a large class of random operators, either for $I=\mathbb{R}$ in the case of sufficiently large $\lambda$ (large disorder) or at the band edges for arbitrary disorder $\lambda > 0$ (spectral extrema). The following list shows exemplary some results where Ineq.~\eqref{eq:fmb} has been proven for the discrete alloy-type model studied in this note. 
\begin{description}
 \item[\cite{AizenmanM-93,Aizenman-94,Graf-94}] consider the case $u = \delta_0$ and $\rho$ bounded and compactly supported (large disorder and spectral extrema).
\item[\cite{AizenmanENSS-06}] consider $u \geq 0$ compactly supported and $\rho$ bounded and compactly supported. They consider the continuum alloy-type model (large disorder and spectral extrema).
\end{description}
Just as in the multiscale analysis, the existing proofs of Ineq.~\eqref{eq:fmb} for non-negative $u$ are not directly applicable if the single-site potential changes its sign. 
For this reason, the fractional moment bound has been established more recently in the non-monotone case.
\begin{description}
\item[\cite{ElgartTV-11}] consider compactly supported $u$ with fixed sign at the vertex-boundary of its support, and a bounded and compactly supported densities $\rho$ (large disorder).
\item[\cite{ElgartSS-12}] assume a covering condition on the compactly supported single-site potential $u\colon \ZZ^d\to \RR$ 
and that the measure $\mu$ is $\tau$-regular and has a finite $q$-moment for some $q > 0$ (large disorder).
\end{description}
\par
Let us now discuss the positivity Assumption (Pos) on the IDS from Theorem~\ref{thm:result1}. 
First of all it is noteworthy that if $N$ is differentiable at $E_0$ and its derivative $n(E_0)$ is positive, then Assumption (Pos) is satisfied for $E_0$ and $\kappa = 0$, 
and one recovers from Theorem~\ref{theorem:poisson} the classical result of \cite{Minami-96} on Poisson statistics, see \cite{GerminetK-11b}. 
Indeed, since we have an optimal Wegner estimate \cite{Veselic-10b} under Assumption \ref{ass:circulant}, the IDS is Lipschitz continuous and hence its derivative, 
the DOS, exists for almost all $E \in \mathbb{R}$. 
The positivity of the DOS is known for the i.i.d.\ Anderson model where $u = \delta_0$. 
\par
The first derivation of the finiteness and positivity of the DOS
was given by Wegner in \cite{Wegner-81}. The argument is on the physical level of rigor, 
but many ideas of the paper have been later given a mathematical justification. At the time of the publication of \cite{Wegner-81} 
the physics community was speculating whether one can identify the mobility edge by some unusual behavior (either vanishing or divergence) of the DOS. These predictions have been discarded by 
\cite{Wegner-81}. 
\par
A precise mathematical justification for the positivity of the DOS 
was achieved in \cite{HislopM-08}, where the authors prove that if the density $\rho$ is essentially bounded away from zero on some interval $[W_- , W_+]$, then for every $\delta > 0$ small enough there exists a strictly positive constant $C_\delta$ such that $n(E) \geq C_\delta$  for Lebesgue-almost all $E \in [-2d + W_- + \delta , 2d + W_+ - \delta]$. A similar result was obtained earlier in the unpublished Ph.D.\ thesis \cite{Jeske-92} supervised by Werner Kirsch. 
For the discrete alloy-type model with a non-trivial single-site potential it is an open question under which conditions on the single-site potential $u$ an the measure $\mu$ 
one can verify the positivity condition (Pos) on the density of states measure.
\paragraph{Acknowledgment} Our work was stimulated by the talks at and discussions around the workshop
\emph{Mathematiques des systeme quantiques desordonnes} organized by H. Boumaza and F.~Klopp
in May 2012 in Paris. The authors are grateful to J.-M.~Combes, F.~Germinet, P.~Hislop,
A.~Klein, F.~Klopp, and F.~Nakano for comments received in response to the transmission of this preprint.
F.~Klopp has furthermore informed us that he has related results to ours. M.T. would like to thank Wilfried Weinelt for stimulating discussions.
\appendix
\section{An explicit Sobolev imbedding}
First we prove an inequality which was used in the proof of Theorem~\ref{thm:result1}.
If $\alpha = (\alpha_1 , \alpha_2)$ is an $2$-tupel of non-negative integers, we denote by 
\[
D^\alpha = \frac{\partial^{\alpha_1}}{x_1^{\alpha_1}} \frac{\partial^{\alpha_2}}{x_1^{\alpha_2}}
\]
the $\alpha$-th weak derivative.
\begin{lemma} \label{lemma:sobolev}
 Let $f \in W^{2,1} (\mathbb{R}^2)$. Then $f \in L^{\infty} (\mathbb{R}^2)$ and
\begin{equation} \label{eq:sobolev}
 \lVert f \rVert_\infty \leq \frac{1}{4} \lVert D^{(1,1)} f \rVert_1 .
\end{equation}
\end{lemma}
\begin{proof} 
The Sobolev imbedding theorem \cite[Theorem 4.12]{Adams-75} tells us that $W^{2,1} (\mathbb{R}^2) \hookrightarrow L^{\infty} (\mathbb{R}^2)$, i.e.\ $W^{2,1} (\mathbb R^2) \subset L^\infty (\mathbb{R}^2)$ and there exists a constant $c$ such that 
\[
\lVert  f \rVert_\infty \leq c \lVert f \rVert_{2,1} = c \sum_{\lvert \alpha \rvert \leq 2} \lVert \partial^\alpha u \rVert_{1} . 
\]
This gives us the first part of our lemma. To prove Ineq.~\eqref{eq:sobolev} we first consider functions $f \in C_{\rm c}^2 (\mathbb{R}^2) \subset W^{2,1} (\mathbb{R}^2)$. By the fundamental theorem of calculus we have for all $x = (x_1 , x_2) \in \mathbb{R}^2$ the four equalities
\begin{align*}
 f(x) &= \int_{-\infty}^{x_1} \int_{-\infty}^{x_2} D^{(1,1)} \! f (x_1' , x_2') \mathrm{d}x_2' \mathrm{d}x_1'
= -\int_{x_1}^{\infty} \int_{-\infty}^{x_2} \! D^{(1,1)} f (x_1' , x_2') \mathrm{d}x_2' \mathrm{d}x_1'  \\
&= -\int_{-\infty}^{x_1} \int_{x_2}^{\infty} \! D^{(1,1)} f (x_1' , x_2') \mathrm{d}x_2' \mathrm{d}x_1'
= \int_{x_1}^\infty \int_{x_2}^{\infty} \! D^{(1,1)} f (x_1' , x_2') \mathrm{d}x_2' \mathrm{d}x_1' .
\end{align*}
We take absolute value of each equality and apply the triangle inequality to get four inequalities. Adding up these inequalities we obtain
\[
 4 \lvert f(x) \rvert \leq \int_{\mathbb R} \int_{\mathbb R} \lvert D^{(1,1)} f(x_1 , x_2) \rvert \mathrm{d} x_2 \mathrm{d} x_1 .
\]
This proves 
\begin{equation} \label{eq:proof1}
\lVert f \rVert_\infty \leq \frac{1}{4} \lVert D^{(1,1)} f \rVert_1 \quad \text{for $f \in C_{\rm c}^2 (\mathbb{R}^2)$}. 
\end{equation}
Let now $f \in W^{2,1} (\mathbb R^2)$. Since $C_{\rm c}^2 (\mathbb R^2)$ is dense in $W^{2,1} (\mathbb R^2)$ there is a sequence $f_n \in C_{\rm c}^2 (\mathbb R^2)$, $n \in \mathbb N$, which converges to $f$ in $W^{2,1} (\mathbb{R}^2)$. The Sobolev imbedding theorem tells us that $f_n$ converges to $f$ in $L^{\infty} (\mathbb R^2)$. The result now follows by letting $n$ tend to infinity in Ineq.~\eqref{eq:proof1} with $f = f_n$.
\end{proof}

\end{document}